\documentclass[12pt]{amsart}
\usepackage{geometry} 
\usepackage{pgfplots}
\usepackage{psfrag}
\usepackage{amsmath}
\usepackage{calc}
\usepackage{systeme}
\usepackage{amsfonts}
\usepackage{amsthm}
\usepackage{algorithm2e}
\usepackage[applemac]{inputenc}
\usepackage{bbold}
\usepackage[numbers]{natbib}

\author{%
  Charles Bertucci $^1$, Jean-Michel Lasry $^1$, 
 Pierre-Louis Lions$^{1,2}$
  }

\newtheorem{Theorem}{Theorem}
\newtheorem{Lemma}{Lemma}
\newtheorem{Rem}{Remark}

\newtheorem{Prop}{Proposition}

\usepackage[foot]{amsaddr}

\title{Some remarks on Mean Field Games}
\thanks{This work is supported by a grant from the Fondation CFM pour la recherche.\\
$^1$ :Universit\'e Paris-Dauphine, PSL Research University,UMR 7534, CEREMADE, 75016 Paris, France\\
$^2$ : Coll\`ege de France, 3 rue d'Ulm, 75005, Paris, France}
\date{} 

\begin{document}
\maketitle
\begin{abstract}
We study in this paper three aspects of Mean Field Games. The first one is the case when the dynamics of each player depend on the strategies of the other players. The second one concerns the modeling of "noise" in discrete space models and the formulation of the Master Equation in this case. Finally, we show how Mean Field Games reduce to agent based models when the intertemporal preference rate goes to infinity, i.e. when the anticipation of the players vanishes.
\end{abstract}
\tableofcontents
\newpage
\section*{Introduction}
This paper presents some results on three topics in the theory of Mean Field Games (MFG). It is divided in three parts, independent of each other for the most part. We present first some results on the uniqueness of solutions for MFG in the case of coupled dynamics. Then, we introduce MFG models with noise in the discrete state space case. We give some conditions for the problem to be well-posed. Finally we discuss the limit of MFG models as the intertemporal preference rate of the players goes to infinity. Those three subjects are described in details later on in each section. We briefly recall now a few well known facts on MFG.\\
\\
MFG have been introduced by the last two authors in \citep{lasry2006jeux, lasry2006jeux2, lasry2007mean} and the field has known a tremendous development in the past ten years. MFG models address differential games involving an infinite number of indistinguishable players. In the absence of a common noise, to find a Nash equilibrium of the game reduces to solve a system of forward backward partial differential equations (PDE). This system consists in a backward Hamilton-Jacobi-Bellman equation, which is solved by the value function of the differential game a generic player is facing, and a forward Fokker-Planck equation which models the evolution of the population. Namely in the case of a game set on the $d$ dimensional torus $\mathbb{T}^d$ with a time horizon $T$, the MFG reduces to systems of the following kind

\begin{equation}\label{MFG}
\begin{cases}
-\partial_t u - \nu \Delta u + H(x, \nabla_x u, m) = 0 \text{ in } (0,T) \times \mathbb{T}^d;\\
\partial_t m - \nu \Delta m - div(D_pH(x, \nabla_x u, m)m) = 0 \text{ in } (0,T)\times \mathbb{T}^d;\\
u(T) = G(x, m) ; m(0) = m_0 \text{ in } \mathbb{T}^d;
\end{cases}
\end{equation}
where $H$ is the Hamiltonian of the optimal control problem for the players, $G$ the terminal cost of this problem and $m_0$ the initial distribution of players. The function $u$ represents here the value function for a generic player and $m$ is the density of players. Uniqueness (under suitable monotonicity conditions) and existence have been proved for this system in \citep{lasry2007mean, lions2007cours}. Numerical methods are being developed, let us cite \citep{achdou2010mean,briceno2016proximal} for examples of this growing literature. Let us also mention the questions of long time average \citep{cardaliaguet2012long, cardaliaguet2017long} or learning \citep{cardaliaguet2015learning}. In some very particular cases (the so-called "potential case") solutions of the MFG system can be obtain from a PDE optimal control problem \citep{benamou2017variational}. This approach is particularly successful to find weak solutions of (\ref{MFG})  \citep{cardaliaguet2015second}. This particular class of MFG is also somewhat similar to the optimal transport theory. Moreover a major axis of development of MFG is the probabilistic interpretation of such differential games. We refer to \citep{carmona2013probabilistic, lacker2015mean} for examples of such a point of view and to \citep{carmona2017probabilistic, carmona2017probabilistic2} for a complete presentation of this probabilistic approach. Let us also recall that MFG are not limited to continuous control problem and that other setting have been studied like optimal stopping \citep{bertucci2017optimal, nutz2016mean, carmona2016mean} or impulse control \citep{bertucci2018fokker} for instance. Moreover we insist that in the presence of a common noise, the reduction of the MFG to the system (\ref{MFG}) collapses and the study of the master equation is crucial for the understanding of the MFG. We refer the reader to \citep{cardaliaguet2015master, lions2007cours} for a detailed study of the master equation and to \citep{cardaliaguet2017long, delarue2018fluctuations} for examples of applications.

\section{ Uniqueness in strongly coupled Mean Field Games}
\subsection{Formulation of the problem}
In this section, we are interested in MFG in the case when the coupling between the players depend on their strategies. We refer to \citep{lions2007cours} for details on this question. In the case without common noise, the study of this strong coupling is the study of  MFG system in which the Hamiltonian of the problem now depends on the measure associated to the distribution of the strategies of the players. We work here on the $d$ dimensional torus $\mathbb{T}^d$ and the system we are interested in is :
\begin{equation}\label{eqcarda}
\begin{cases}
-\partial_t u - \nu \Delta u + H(x,\nabla u, m , µ) = 0  \text{ in } (0,T) \times \mathbb{T}^d;\\
\partial_t m - \nu \Delta m - div(D_pH(x,\nabla u , m , µ)m) = 0 \text{ in } (0,T)\times \mathbb{T}^d;\\
m(0) = m_0 ; u(T) = \phi(m) \text{ in } \mathbb{T}^d;\\
\end{cases}
\end{equation}
where $H : \mathbb{T}^d \times \mathbb{R}^d \times \mathcal{P}(\mathbb{T}^d)\times  \mathcal{P}(\mathbb{R}^d) \to \mathbb{R}$ denotes the Hamiltonian associated with the optimal control problem the players are facing. It is assumed to be smooth and $D_pH$ stands for the derivative with respect to its second variable. We denote by $u$ the value function of the players, $m$ the measure which describes their distribution in the state space $\mathbb{T}^d$ and $µ$ the measure associated to the controls of the players. Let us remark that because $µ$ is the measure of the optimal controls used by the players, it satisfies the following relation :
\begin{equation}\label{fixpointmu}
µ = (-D_pH(x, \nabla u , m , µ))_{\# }m;
\end{equation}
where $G_{\#} m$ stands for the image measure of $m$ by the measurable application $G$. Existence of solutions of such a system was proven for particular Hamiltonians in \citep{Gomes2014existence} ; we refer to \citep{cardaliaguet2016mean} for a more general result of existence for (\ref{eqcarda}) and to \citep{carmona2015probabilistic, carmona2017probabilistic} for the study of such a coupling in a probabilistic setup. In general, uniqueness of such solutions is not known.

Before presenting the main result of this part, we make some preliminaries observations. More precisely, we state that the quantity
\begin{equation}\label{quantity}
\int_{\mathbb{T}^d}  \nabla u m
\end{equation}
is conserved in time under the assumption that the Hamiltonian $H$ is of the form : $H(\nabla u,m)(x) = \tilde{H}(\nabla u(x)) - f(m(x))$. This result is quite simple but we believe it is worth mentioning as it does not seem to appear in the literature. And it shows that there may be properties enjoyed by the solutions which may play a role in uniqueness properties. Moreover, it will be useful for the study of the example we give at the end of this part.

\begin{Lemma}\label{conservation}
Assume that $H$ does not depend on the space variable $x$ and that it is of the form $H(\nabla u,m)(x) = \tilde{H}(\nabla u(x)) - f(m(x))$, then for any solution $(u,m)$ of (\ref{MFG}) the quantity 
\begin{equation}\label{quantity}
\int_{\mathbb{T}^d}  \nabla u m
\end{equation}
does not depend on time.
\end{Lemma}
\begin{proof}
We only prove the statement in the case of a smooth solution $(u,m)$ as it easily extends to weaker notions of solutions. Multiplying by $\partial_{x_i} m$ the equation satisfied by $u$ and integrating the PDE satisfied by $u$ on $(0,t) \times \mathbb{T}^d$, we obtain
\[
\begin{aligned}
0  &= \int_0^t \int_{\mathbb{T}^d} (-\partial_t u - \nu \Delta u + \tilde{H}(\nabla u) - f(m))\partial_{x_i} m ;\\
& \begin{aligned} =\int_0^t \int_{\mathbb{T}^d} (-\partial_t m &+ \nu \Delta m) \partial_{x_i} u + \tilde{H}( \nabla u)\partial_{x_i} m - \int_{\mathbb{T}^d}                  f(m)\partial_{x_i} m \\
                      &+ \int_{\mathbb{T}^d} (u(0) \partial_{x_i} m_0 - u(t) \partial_{x_i} m(t)) ;\end{aligned}\\
& = \int_0^t \int_{\mathbb{T}^d}  -div(D_p\tilde{H}(\nabla u)m) \partial_{x_i} u +  \tilde{H}( \nabla u) \partial_{x_i} m + \int_{\mathbb{T}^d} (u(0) \partial_{x_i} m_0 - u(t) \partial_{x_i} m(t)).
\end{aligned}
\]
We then deduce by integrating by parts all the terms of the last line :
\[
\begin{aligned}
0 &=  \int_0^t \int_{\mathbb{T}^d} m D_p\tilde{H}(\nabla u) \cdot \partial_{x_i} \nabla u - D_p\tilde{H}(\nabla u) \cdot m \partial_{x_i} \nabla u  - \int_{\mathbb{T}^d} (m_0 \partial_{x_i} u(0) - m(t) \partial_{x_i} u(t));\\
 & =  -\int_{\mathbb{T}^d} (m_0 \partial_{x_i} u(0) - m(t) \partial_{x_i} u(t)).
\end{aligned}
\]
\end{proof}
\begin{Rem}
We thus see that the invariance by translation of the Hamiltonian implies that (\ref{quantity}) is independent of time. Such a quantity is closely related to the average control of the players. Hence such a lemma leads us to think that one can obtain information on the average control using assumptions on the space dependence of the Hamiltonian.
\end{Rem}
As we are led to distinguish local and global dependence of the Hamiltonian in $\nabla u$, we introduce the following notation :
\begin{itemize}
\item We note $G(f)$ when the dependence of the function $G$ in the function $f$ is local, i.e. $G(f)(x) = G(f(x))$.
\item We note $G(f(\cdot))$ for a general dependence of the function $G$ in the function $f$.
\end{itemize}
We shall only use this notation in the case when there might be an ambiguity.

\subsection{A uniqueness result}
We present here a uniqueness result for the system :
\begin{equation}\label{ourmodel2}
\begin{cases}
-\partial_t u - \nu \Delta u + H(\nabla u) - G(\nabla u ( \cdot), m (\cdot)) \cdot \nabla u = f(u,m) \text{ in } (0,T) \times \mathbb{T}^d;\\
u(T) = \phi(m(T)) \text{ in } \mathbb{T}^d;\\
\partial_t m - \nu \Delta m - div(m(D_p H ( \nabla u) - G(\nabla u ( \cdot), m (\cdot)) )) = g(u,m) \text{ in } (0,T) \times \mathbb{T}^d;\\
m(0) = m_0 \text{ in } \mathbb{T}^d.
\end{cases}
\end{equation}
Such a system models a MFG in which the players are "pushed" by the mean field control of the other players. We recall the usual conditions under which we expect uniqueness to hold for solutions of (\ref{ourmodel2}) when $G = 0$:
\begin{itemize}
\item The hamitlonian $H$ is convex.
\item The terminal cost $\phi$ is monotone.
\item There exist $\alpha$ and $\beta$ such that $(g,f) +(\alpha m, \beta u)$ is monotone in $(u,m)$.
\item Either $H$ is strictly convex, $(g,f) +(\alpha m, \beta u)$ is strictly monotone in $(u,m)$, $\phi$ is strictly monotone or $(g,f)$ is local in $(u,m)$.
\end{itemize}
We say that the hypothesis $(H1)$ is satisfied if those four requirements are satisfied. We also recall the notion of monotonicity for the sake of completeness. To define this notion properly one usually needs to make precise the functions spaces on which $f$ and $g$ are defined and take their values. To fix ideas, we recall the notion of monotonicity in the case of maps of $L^2$ : a map $F$ from $L^2(\Omega)$ into itself is said to be (resp. strictly) monotone if for any $u_1, u_2 \in L^2(\Omega)$ :
\[
\int_{\Omega}(F(u_1) - F(u_2))(u_1 - u_2) \geq 0 \text{ (resp } > 0 \text{ if } u_1 \ne u_2 \text{ )}
\]
We define by $\mathcal{T}$ the group of transformation :
\[
\mathcal{T} := \{ (v, m) \rightarrow (v \circ \tau_x, m\circ \tau_x) | \tau_x \text{ is the translation by } x \in \mathbb{T}^d\}
\]
We are now able to state the following result :
\begin{Theorem}\label{uniquecontrol}
Let us assume that $G$ and $\phi$ are invariant by translations, i.e. that they do not depend on the space variable $x$ and that $G$ is invariant under the group of transformation $\mathcal{T}$. Then under the usual condition that $(H1)$ is satisfied, there exists at most one solution $(u,m)$ of (\ref{ourmodel2}).
\end{Theorem}
\begin{proof}
We denote by $(u_1, m_1)$ and $(u_2, m_2)$ two solutions of the system $(\ref{ourmodel2})$. For $i \in \{1, 2\}$ we define $(\tilde{u}_i, \tilde{m}_i)$ by
\[
(\tilde{u}_i, \tilde{m}_i)(t,x) = (u,m) (t, x + \int_0^t G(\nabla u_i(s), m_i(s))ds).
\]
Because $G$ and $\phi$ are invariant by translations, we deduce that for $i \in \{1,2\}$, $(\tilde{u}_i, \tilde{m}_i)$ is a solution of the MFG system :
\[
\begin{cases}
-\partial_t \tilde{u}_i - \nu \Delta \tilde{u}_i + H(\nabla \tilde{u}_i) = f(\tilde{u}_i, \tilde{m}_i) \text{ in } (0,T) \times \mathbb{T}^d;\\
\partial_t \tilde{m}_i - \nu \Delta \tilde{m}_i - div(D_p H(\nabla \tilde{u}_i) \tilde{m}_i) = g(\tilde{u}_i, \tilde{m}_i) \text{ in } (0,T) \times \mathbb{T}^d;\\
\tilde{m}_i(0) = m_0 ; \tilde{u}_i (T) = \phi(\tilde{m}_i(T)) \text{ in } \mathbb{T}^d.
\end{cases}
\]
Let us note that the assumption that $H$ does not depend explicitly on the space variable is crucial. Because the assumption (H1) holds, there is a unique solution of this system and $(\tilde{u}_1, \tilde{m}_1) = (\tilde{u}_2, \tilde{m}_2)$. Now let us remark that because $G$ is invariant by the transformation which sends $(u,m)$ on $(\tilde{u}, \tilde{m})$ we obtain that 
\[
G(\nabla u_1, m_1) = G(\nabla u_2, m_2) = G(\nabla \tilde{u}, \tilde{m}).
\]
Thus, $(u_1, m_1) = (u_2, m_2)$ and there is a unique solution of (\ref{ourmodel2}).
\end{proof}

\subsection{A simple example}
The previous proof relies strongly on the form of the Hamiltonian and on the invariance by translations of the problem. We now discuss to what extent those assumptions are needed. First, it is very natural to allow the measure of controls $µ$ in (\ref{eqcarda}) to appear in the Hamiltonian only through terms like $G(\nabla u, m)$ because we are interested in a mean field interaction with the control of the other players. Thus terms of the form 
\[
G(\nabla u(\cdot), m(\cdot)) = \int_{\mathbb{T}^d}g(\nabla u(y))m(y)dy
\]
seem to be quite general for applications. See the example studied in \citep{cardaliaguet2016mean} for example. Moreover such a term satisfies the assumption of theorem \ref{uniquecontrol}. However it is true that the assumption that this term is linear in the "$p$" variable for the Hamiltonian is rather strong (i.e. that $G$ appears in the Hamiltonian only through $G(\nabla u, m) \cdot \nabla u$).\\
\\
Concerning the question of the invariance by translation, the following example shows that this assumption can be extended in some cases. But it also shows that there is a complex interaction between the uniqueness and the type of dependence of the Hamiltonian in the control of the other players. Let us introduce the following system which satisfies all the assumptions of  the theorem \ref{uniquecontrol} except for the fact that the terminal cost $\phi$ now depends on the space variable.

\begin{equation}\label{example1}
\begin{cases}
-\partial_t u -\nu \Delta u+ \frac{1}{2}| \nabla u |^2 -\lambda (\int \nabla u m) \cdot \nabla u = 0 \text{ in } (0,T)\times \mathbb{T}^d; \\
\partial_t m -\nu \Delta m + div( (\nabla u - \lambda (\int \nabla u m)) m) = 0 \text{ in } (0,T)\times \mathbb{T}^d; \\
u(T) = \phi(x) ; m(0) = m_0(x) \text{ in } \mathbb{T}^d.
\end{cases}
\end{equation}
Even though this section is not particularly concerned with the question of existence, let us mention that a solution of such a system exists as soon as $\phi$ is a Lipschitz function. The following lemma gives a general property satisfied by the solutions of (\ref{example1}). Proposition \ref{example2} details a precise example of a function $\phi$  which yields uniqueness even though it is not constant :

\begin{Lemma}\label{eqA}
Let us denote by $(u,m)$ a solution of (\ref{example1}), we define $A = \int_{\mathbb{T}^d} \nabla u m$ and we denote by $u_0$ the solution of the HJB equation
\[
\begin{cases}
-\partial_t u_0 -\nu \Delta u_0 + \frac{1}{2} |\nabla u_0|^2 = 0 ;\\
u_0(T) = \phi.
\end{cases}
\]
Then the following holds
\begin{equation}\label{eqAprop}
A = \int_{\mathbb{T}^d} \nabla u_0( x + \lambda A T,0) m_0(x) dx.
\end{equation}
\end{Lemma}
\begin{proof}
Let us remark that the proof of lemma \ref{conservation} applies also for (\ref{example1}), hence $(u,m)$ satisfies :
\[
\frac{d}{dt} \int_{\mathbb{T}^d} \nabla u m = 0.
\]
Thus $A$ is well defined. Now let us remark that a simple change of variable yields:
\[
u(t,x) = u_0(t, x + \lambda A ( T - t)).
\]
Hence, evaluating $\int_{\mathbb{T}^d} \nabla u m$ at $t= 0$ we obtain that
\[
A = \int_{\mathbb{T}^d} \nabla u_0( x + \lambda A T,0) m_0(x) dx.
\]
\end{proof}
Each solution $A$ of this equation produces a solution for the system (\ref{example1}). Hence if $\phi$ and $m_0$ are such that the previous equation admits a unique solution then we have proved uniqueness. Such a situation is possible if we obtain a precise bound on $\nabla u_0$ for instance. We give now an example of such a situation :

\begin{Prop}\label{example2}
If $D^2\phi(x) \leq c \text{  }Id$ holds for some $c >0$ in the sense of distributions, then uniqueness in (\ref{example1}) holds if 
\[
\lambda < \frac{1 + cT}{cT}.
\]
\end{Prop}

\begin{proof}
Under those assumptions, in view of classical semi-concavity estimates for HJB equations, the Hessian matrix $D^2 u_0$ of $u_0$ also satisfies in the sense of distributions for any $0 \leq t \leq T $:
\[
D^2u_0(t) \leq \frac{c}{1 + c(T-t)}Id.
\]
Hence the equation (\ref{eqAprop}) satisfied by $A$ admits a unique fixed point if  :
\[
\lambda < \frac{1 + cT}{cT};
\]
because then $A - \int_{\mathbb{T}^d} \nabla u_0( x + \lambda A T,0) m_0(x) dx$ is a strictly increasing function of $A$.
\end{proof}

\begin{Rem}
It is clear from (\ref{eqAprop}) that without an assumption on $\phi$, uniqueness may not hold. Moreover such a proof can be adapted to a case in which we add a dependence $f(m)$ in the HJB equation.
\end{Rem}
\begin{Rem}
Let us note that the condition in this proposition can be interpreted as the bigger $c$ is (i.e. the more the system depends on the space variable $x$), the more the effect of the average strategy (the parameter $\lambda$) needs to be small. Also, if $\lambda < 1$, then the condition is satisfied independently of $c$, thus the condition $\lambda < 1$ yields uniqueness even in the case $c = + \infty$, that is for general $\phi$.
\end{Rem}

\section{Common noise in discrete state space MFG}\label{sec : mfgn}
The addition of a common noise in the MFG setting remains one of the most important questions in the MFG theory. Even though \citep{cardaliaguet2015master, lions2007cours} provide a detailed study of the so-called master equation, several problems remain open. We can cite for example wether or not there exist weaker regularity assumptions than the one needed in \citep{cardaliaguet2015master}, or the possibility to give an extended sense of solutions which yields existence results in the non-monotone case. We present here a discrete state space master equation which models the presence of a common noise in the MFG. We first recall some facts on the master equation without common noise before explaining what we mean by common noise in the discrete state space. We then prove the well-posedness of such an equation as well as some limit equations which can be obtained from our model.

\subsection{The master equation for a discrete state space}
In a discrete state space, without a common noise, the infinite dimensional master equation reduces to a first order non conservative hyperbolic system of the following form (see \citep{lions2007cours} for instance) : 
\begin{equation}\label{MFGn}
\begin{cases}
\partial_t U(t,x) + (F(x, U) \cdot \nabla)U  = G(x,U) \text{ in } (0,t_f)\times \mathbb{R^d};\\
U(0,x) = U_0(x) \text{ in } \mathbb{R^d}.
\end{cases}
\end{equation}
Let us recall that in such a system, $U$ is a map from $(0,t_f)\times \mathbb{R}^d$ to $\mathbb{R}^d$, $F$ and $G$ are applications from $\mathbb{R}^{2d}$ to $\mathbb{R}^d$. The time $t_f$ is the final time (fixed in $(0,\infty)$) and the initial condition is $U_0$. We also recall that the system has to be understood as follows : for all $1 \leq i \leq d$, the $ith$ component $U_i$ of $U$ satisfies :
\[
\begin{cases}
\partial_t U_i(t,x) + F(x, U) \cdot \nabla U_i  = G_i(x,U) \text{ in } (0,t_f)\times \mathbb{R^d}\\
U_i(0,x) = (U_0)_i(x) \text{ in } \mathbb{R^d}.
\end{cases}
\]
The MFG interpretation of (\ref{MFGn}) is that $x$ represents the number of players in the $d$ possible states, meaning that $x_1$ is the number of players in the first state, $x_2$ the number of players in the second and so on. The interpretation of the map $U$ is that the $ith$ component $U_i(t,x)$ of $U(t,x)$ is the value of the optimal control problem for a generic player in the state $i$, with the repartition of other players being $x$ at the time $t$. Let us note that contrary to the classical MFG system (\ref{MFG}) or to the master equation presented in \citep{cardaliaguet2015master}, the equation is written here forward in time, instead of backward (i.e. we reversed time with $t \to T- t$).\\
\\
We also recall that there exists a system of characteristics for (\ref{MFGn}). It is defined for any $x_0 \in \mathbb{R}^d$ by 
\begin{equation}\label{chara}
\begin{cases}
\frac{d}{ds}V(s) = G(Y(s),V(s));\\
\frac{d}{ds} Y(s) = F(Y(s), V(s));\\
Y(0) = x_0 ; V(0) = U_0(x_0).
\end{cases}
\end{equation}
Using the previous characteristics, we can define a solution $U$ of (\ref{MFGn}) by :
\[
U(t, Y(t)) = V(t);
\]
One can show (\citep{lions2007cours}) that it is possible to define a solution of (\ref{MFGn}) using the characteristics under a monotonicity assumptions. The system (\ref{chara}) is the analogous of the usual MFG system (\ref{MFG}) in a discrete state space. Thus, many results on (\ref{MFGn}) can be proven using this system of characteristics (\ref{chara}). However, addressing the regularity and the well-posedness of (\ref{MFGn}) without using (\ref{chara}) leads to techniques and proofs which are adaptable to the addition of a common noise, whereas the use of the system (\ref{chara}) collapses in the presence of a common noise.\\
\\
As the notion of monotonicity is crucial in all the results we present in this section, we define here what we mean by monotone. A map $V$ from $\mathbb{R}^d$ into itself is said to be monotone if 
\[
\forall x, y \in \mathbb{R}^d, < V(x) - V(y), x - y > \geq 0.
\]
The map $V$ is said to be $\alpha$ monotone if
\[
\forall x, y \in \mathbb{R}^d, < V(x) - V(y), x - y > \geq \alpha |x - y|^2.
\]
Here, $< x, y> = x \cdot y$ denotes the usual scalar product between $x$ and $y$. 

\subsection{The master equation in the presence of noise}
We propose here a master equation in finite dimension with additional terms which model jumps which can occur in the population. Those jumps can be random and thus can be assimilated to noise in the MFG. The jumps are described by an application $T$. We give two examples of such maps $T$ to fix ideas, and explain how the jumps have to interpreted, depending on the nature of the noises.
\begin{itemize}
\item If the map $T$ is defined by :
\[
T(x) = (x_2,x_1, x_3,...,x_d);
\]
then the map $T$ models jumps which only concerns players in the first or second state. The players in the first state are sent to the second one and vice versa. If all the players face this jump at the same time, then all the players in the first state and the ones in the second one are exchanged. If only a proportion $0 \leq \theta \leq 1$ of the players in each states jump simultaneously, then the state of the population after the jump is given by $\theta T(x) + (1- \theta)x$, if $x$ is the state of the population before.

\item If the map $T$ is defined by :
\[
T(x) = (\frac{x_1}{2} + \delta , x_2 + \frac{x_1}{2},x_3, .., x_d);
\]
where $\delta > 0$ is fixed, then if a player is in the first state, it has a probability one half to jump to the second state. This jump is naturally associated with a structure of noise, independently of the randomness that may exist concerning the players this jump affects or the time at which it occurs. If a proportion $0 \leq \theta \leq 1$ of the players in all states are simultaneously affected by this jump, then a fraction $\frac{\theta}{2}$ of the players in the first state jump in the second one and an amount of players corresponding to $\theta \delta$ is added to the game in the first state.
\end{itemize}
The noise structures modeled by this kind of terms are numerous. Indeed with this second example, we show that there can be randomness in the fact that a player will jump or not, just as there can be randomness on the state in which it is going to jump. There can also be randomness on the times at which those jumps occur. The correlations between those times can be quite general as we shall now see through some examples of master equations.

We now make precise how jumps are taken into account in the master equation, depending on the correlations of the jumps. We give three examples of master equation which models jumps occurring at deterministic times, or when the times of the jumps are random but common to every players and finally when they are random and independent identically distributed (iid) for the players. In those examples, $T$ is a differentiable map from $\mathbb{R}^d$ to itself.\\
\\
If there is only one jump characterized by $T$ which occurs at the deterministic time $t_1 \in (0,t_f)$, then the master equation is :
\begin{equation}\label{med}
\begin{cases}
\partial_t U(t,x) + (F(x, U) \cdot \nabla)U = G(x,U) \text{ in } (t_1,t_f)\times \mathbb{R^d};\\
U(t_1^+,x) = (DT(Tx))^*U(t_1^-, T(x)) \text{ in } \mathbb{R}^d;\\
\partial_t U(t,x) + (F(x, U) \cdot \nabla)U = G(x,U) \text{ in } (0,t_1)\times \mathbb{R^d};\\
U(0,x) = U_0(x) \text{ in } \mathbb{R^d};
\end{cases}
\end{equation}
where $t_f> 0$ is the time horizon and $(DT)^*$ stands for the adjoint of the differential of the map $T$.\\
\\
If at random times given by an exponential law of parameter $\lambda$, all the players jump according to $T$, then the master equation is :
\begin{equation}\label{MFGnl}
\begin{cases}
\partial_t U(t,x) + (F(x, U) \cdot \nabla)U + \lambda (U - (DT(T(x)))^*U(t, T(x))) = G(x,U) \text{ in } (0,t_f)\times \mathbb{R^d};\\
U(0,x) = U_0(x) \text{ in } \mathbb{R^d};
\end{cases}
\end{equation}
where $\lambda$ is the parameter of the Poisson process which describes the jumps. In (\ref{MFGnl}), the term 
\[
\lambda (U - (DT(Tx))^*U(t, Tx))
\]
is understood as the players anticipating the noise. The proper derivation of such a term is technical and we do not enter in such calculations here. We refer to \citep{conze1991system, conze1993borrowing} for a proper derivation of this term in a different setting, when $T$ is a linear map.\\
\\
If all the players jump according to $T$ at times given by iid Poisson processes of paramter $\lambda$, then the master equation is given by :
\begin{equation}\label{MFGnli}
\begin{cases}
\begin{aligned}
G(x, U) = &\partial_t U(t,x) + (F(x, U) \cdot \nabla)U \\
&+ \lambda (Id - (DT(x))^*)U(t,x) +\lambda((Id - T(x))x \cdot \nabla) U \\
& \text{ in } (0,t_f)\times \mathbb{R^d};
\end{aligned}\\
U(0,x) = U_0(x) \text{ in } \mathbb{R^d}.
\end{cases}
\end{equation}
In this equation, the drift term $\lambda((Id - T(x))x \cdot \nabla) U $ is interpreted as the effect of the flow of players which are jumping. Indeed, because the jumps are independent and there is an infinity of players, there is always a constant proportion of players which are jumping. This phenomenon induces a dynamic in the population which is taken into account in this drift term.

\begin{Rem}
Let us remark that a master equation can be written in a similar way if some of the above three phenomena happen together in the MFG.
\end{Rem}
\subsection{Propagation of monotonicity and regularity}
In this section, we shall show that under some assumptions on the form of $T$ , the solutions of (\ref{med}), (\ref{MFGnl}) and (\ref{MFGnli}) have some monotonicity properties. We then show how we can establish regularity properties from the propagation of the monotonicity.

\begin{Prop}\label{monotone1}
Assume that $T$ is affine. Let $U$ be a classical solution of (\ref{med}), (\ref{MFGnl}) or (\ref{MFGnli}). If $(G,F)$ (seen as a map from $\mathbb{R}^{2d}$ into itself) is monotone  and $U_0$ is monotone, then $U$ is monotone for all time.
\end{Prop}

\begin{proof}
The proof of this result follows the arguments in \citep{lions2007cours}. We write the affine map as $T = S + e$ where $e\in \mathbb{R}^d$ and $S$ is a linear map.\\
\\
\textit{The case of (\ref{med}) :}\\
Let $U$ be a classical solution of $(\ref{med})$, we denote by $V$ another classical solution of (\ref{med}) and we define $W$ by
\[
W(t,x,y) = < U(t,x) - V(t,y), x - y >.
\]
Let us note that $W$ satisfies the following equation in $((0,t_1) \cup (t_1, t_f))\times \mathbb{R^d}$:
\[
\begin{aligned}
\partial_t W + F(x,U) \cdot \nabla_x W + F(y,V)\cdot \nabla_y W =& <G(x,U) - G(y,V), x - y> \\
& + < F(x,U) - F(y, V), U(t,x) - V(t,y)>;
\end{aligned}
\]
Let us now take $V= U$. Because $U_0$ is monotone, we deduce that $W(t,x,y) \geq 0$ for all $x,y \in \mathbb{R}^d$. We also deduce from the monotonicity of $(G,F)$ that the right hand side of the previous equation is positive everywhere. Thus we deduce form the maximum principle (lemma \ref{max} in the appendix) that $U$ is monotone ($W$ is positive) in $(0,t_1)$. Now let us remark that the relation satisfied by $U$ at $t_1$ yields for $W$ :
\[
W(t_1^+, x,y) = W(t_1^-, Tx ,Ty).
\]
Thus $W$ is positive at $t_1^+$ and thus $W$ is positive for all time $t \in (0,t_f)$, by repeating the same argument.\\
\\
\textit{The case of (\ref{MFGnl}):}\\
Let $U$ be a classical solution of $(\ref{MFGnl})$, we denote by $V$ another classical solution of (\ref{MFGnl}) and we define again $W$ by
\[
W(t,x,y) = < U(t,x) - V(t,y), x - y >.
\]
Let us note that $W$ satisfies the following equation :
\[
\begin{aligned}
\partial_t W +  & F(x,U) \cdot \nabla_x W + F(y,V)\cdot \nabla_y W + \lambda W - \lambda < S^*U(t,Tx) - S^*V(t,Ty) , x - y> \\
& = <G(x,U) - G(y,V), x - y> + < F(x,U) - F(y, V), U(t,x) - V(t,y)>;
\end{aligned}
\]
which can be rewritten in :
\[
\begin{aligned}
\partial_t W +  & F(x,U) \cdot \nabla_x W + F(y,V)\cdot \nabla_y W + \lambda W - \lambda W(t,Tx,Ty) \\
& = <G(x,U) - G(y,V), x - y> + < F(x,U) - F(y, V), U(t,x) - V(t,y)>.
\end{aligned}
\]
By the maximum principle (lemma \ref{max} in the appendix), we deduce that $U$ is monotone for all time. This is a straightforward consequence of the monotonicity of $(G,F)$ and $U_0$, taking $V = U$ in $W$.\\
\\
\textit{The case of (\ref{MFGnli}) :}\\
Let $U$ be a classical solution of $(\ref{MFGnli})$, we denote by $V$ another classical solution of (\ref{MFGnli}) and we define once again $W$ by
\[
W(t,x,y) = < U(t,x) - V(t,y), x - y >.
\]
The appliccation $W$ satisfies the following equation :
\[
\begin{aligned}
\partial_t W +   F(x,U) \cdot &\nabla_x W + F(y,V)\cdot \nabla_y W + \lambda (S- Id)x \cdot \nabla_x W\\
 + \lambda (S- Id&)y \cdot \nabla_y W + \lambda <U - V, (Id - S)(x-y)>   \\
 = <G(x,U)& - G(y,V), x - y> + < F(x,U) - F(y, V), U(t,x) - V(t,y)> \\
+ \lambda <U - &V, (Id - S)(x-y)>;
\end{aligned}
\]
Thus $W$ satisfies :
\[
\begin{aligned}
\partial_t W +   F(x,U) \cdot &\nabla_x W + F(y,V)\cdot \nabla_y W + \lambda (S- Id)x \cdot \nabla_x W\\
 + \lambda (S- Id&)y \cdot \nabla_y W\\
 = <G(x,U)& - G(y,V), x - y> + < F(x,U) - F(y, V), U(t,x) - V(t,y)>;
\end{aligned}
\]
Once again, by the maximum principle (lemma \ref{max} in the appendix), we deduce that $U$ is monotone for all time by taking $V = U$ and using the monotonicity of $U_0$ and $(G,F)$.
\end{proof}
\begin{Rem}
Let us note that in particular, under the assumptions of proposition \ref{monotone1}, there exists at most one classical solution of (\ref{med}), (\ref{MFGnl}) and (\ref{MFGnli}). Indeed uniqueness follows readily from the positivity of $W$ (see \citep{lions2007cours}).
\end{Rem}

As we can see in the proof of the previous statement, the propagation of the monotonicity of $U$ is really a "strong" property of (\ref{med}), (\ref{MFGnl}) and (\ref{MFGnli}), in the sense that it does not depend on $\lambda$, $t_1$ or on $T$. The different terms in each system only reinforce the propagation of monotonicity. 

We now prove that under stronger assumptions on $U_0$ and $(G,F)$, we can deduce from the propagation of monotonicity, some regularity for the system (\ref{MFGnl}). Similar results can be obtain for (\ref{med}) and (\ref{MFGnli}) by similar arguments.

\begin{Theorem}\label{amonotone}
Assume that the map $T$ is affine and that $U_0$, $F$ and $G$ are Lipschitz continuous. Let $U$ be a solution of (\ref{MFGnl}). If both $U_0$ and $(G,F)$  are monotone and either $U_0$ and $G$ are $\alpha$ monotone or $F$ is $\alpha$ monotone (for some $\alpha > 0$), then $U$ is Lipschitz in space, uniformly on $(0, t_f)$, for all $t_f > 0$.
\end{Theorem}

\begin{proof}
We here only prove the result in the case in which both $U_0$ and $(G,F)$ are $\alpha$ monotone. The proof of the complete statement is given in the appendix. This proof also follows the argument in \citep{lions2007cours}. Once again we note the affine map $T = S + e$ with $S$ linear and $e\in \mathbb{R}^d$ constant. We define $W$ and $Z_{\beta}$ by
\[
W(t,x,\xi) = <U(t,x), \xi> ;
\]
\[
Z_{\beta}(t,x,\xi) = <\xi, \nabla_x W(t,x,\xi)> - \beta |\nabla_x W(t,x,\xi)|^2.
\]
We now write the PDE satisfied by $Z_{\beta}$ :

\begin{equation}\label{systZ}
\begin{aligned}
& \partial_t Z_{\beta} + <F(x,\nabla_{\xi} W), \nabla_x Z_{\beta}> + < D_p F (x,\nabla_{\xi} W ) \nabla_{\xi}Z_{\beta} , \nabla_{x} W> - < D_p G(x, \nabla_{\xi}W)\nabla_{\xi}Z_{\beta}, \xi>\\
&+  \lambda(Z_{\beta} - Z_{\beta}(t,Tx, T\xi - e) )\\
& = <\nabla_xG(x, \nabla_x W) \xi, \xi> - < D_pG(x, \nabla_{\xi}W) \nabla_x W, \xi> - < \nabla_x F(x, \nabla_{\xi} W) \nabla_x W, \xi> \\
&+ <D_p F(x, \nabla_{\xi}W) \nabla_x W, \nabla_x W>\\
&- 2 \beta < \nabla_x G(x, \nabla_x W)\xi, \nabla_x W> + 2 \beta < \nabla_x F( \nabla_{\xi} W) \nabla_x W, \nabla_x W>\\
& + \beta \lambda \bigg{(} |\nabla_x W|^2 - 2 < \nabla_x W(t, Tx, T \xi - e), S \nabla_x W> +  | \nabla_x W(t,Tx, T \xi -e)|^2 \bigg{)}. 
\end{aligned}
\end{equation}
For the sake of clarity, we do not detail the calculation which leads to the previous equation, which easily follows from the chain rule. Let us remark that the following holds :
\[
|\nabla_x W|^2 - 2 < \nabla_x W(t, Tx, T \xi - e), S \nabla_x W> +  | \nabla_x W(t,Tx, T \xi - e)|^2 \geq |\nabla_x W|^2 - |S \nabla_x W|^2.
\]
Because $(G,F)$ and $U_0$ are $\alpha$ monotone, we deduce that :
\[
\begin{aligned}
& \partial_t Z_{\beta} + <F(x,\nabla_{\xi} W), \nabla_x Z_{\beta}> + < D_p F (x,\nabla_{\xi} W ) \nabla_{\xi}Z_{\beta} , \nabla_{x} W> - < D_p G(x, \nabla_{\xi}W)\nabla_{\xi}Z_{\beta}, \xi>\\
&+  \lambda(Z_{\beta} - Z_{\beta}(t,Tx, T\xi - e) )\\
&\geq \alpha(|\xi|^2 + |\nabla_x W|^2) + \beta \lambda (|\nabla_x W|^2 - |S \nabla_x W|^2) \\
& - 2 \beta < \nabla_x G(x, \nabla_x W)\xi, \nabla_x W> + 2 \beta < \nabla_x F( \nabla_{\xi} W) \nabla_x W, \nabla_x W>;\\
& \geq \alpha(|\xi|^2 + |\nabla_x W|^2) + \beta \lambda (1 - ||S||^2)|\nabla_x W|^2  \\
& - 2 \beta (||\nabla_x G|| \cdot |\xi | \cdot | \nabla_x W| + || \nabla_x F|| \cdot |\nabla_x W|^2).
\end{aligned}
\]
Hence, if $\beta$ satisfies the inequality (\ref{betacond}), then $Z_{\beta}$ is positive for all time (by lemma \ref{max} in appendix).
\begin{equation}\label{betacond}
\beta \leq \min(\frac{\alpha}{\lambda(||S||^2 - 1)^+ + ||\nabla_x G|| + 2 ||\nabla_x F||}, ||\nabla_x U_0||).
\end{equation}
From the fact that $Z_{\beta}$ is positive for all time, we deduce that :
\[
\forall t, x, \xi : |\nabla_x W(t,x,\xi)| \leq \beta^{-1}|\xi|.
\]
Hence $U$ is $\beta^{-1}$ Lipschitz, uniformly in time.

\end{proof}
\begin{Rem}
Let us notice that the constant $\beta$ involved in the proof only depends on $\alpha$, $t_f$, $\lambda(||DT||^2 - 1)^+$ and on the Lipschitz constants of $U_0$, $F$ and $G$. Thus if the application $S$ is non-expansive, then the constant $\beta$ can be chosen independently of $\lambda$ and $T$, thus we can expect that the sequence of solutions $(U_{\lambda})_{\lambda > 0}$ has a limit as $\lambda$ goes to infinity, extracting a subsequence if necessary.
\end{Rem}
This result emphasis the fact that the monotonicity of the system "implies" the regularity of the solution. Indeed the Lipschitz regularity is enough to prove higher order regularity see \citep{lions2007cours} for instance. Finally we end this discussion on the well-posedness of this system of conservation laws with this straightforward application :

\begin{Theorem}
Let $\mathcal{X}$ be a compact set of affine functions and $µ$ a measure on $\mathcal{X}$. Then, under the assumptions of theorem \ref{amonotone}, there exists a unique smooth monotone solution $U$ of :
\begin{equation*}
\begin{cases}
\partial_t U(t,x) + (F(x, U) \cdot \nabla)U + \lambda \int_{\mathcal{X}}  (U - (DT)^*U(t, Tx))dµ(T) = G(x,U) \text{ in } (0,t_f)\times \mathbb{R^d};\\
U(0,x) = U_0(x) \text{ in } \mathbb{R^d}.
\end{cases}
\end{equation*}
\end{Theorem}

\subsection{Asymptotic differential operators}
In this section, we present how we can derive, from a common noise term, higher order terms in (\ref{MFGnl})(first order and second order terms) which conserve the monotonicity property, and thus the well-posedness of this equation. In the previous section we showed that adding the term ($T$ is assumed to be linear in all the rest of this section) :
\[
\lambda (U - T^* U(T \cdot))
\] 
does not alter the propagation of monotonicity. Formally, if we set $\lambda = \epsilon^{-1}$ and $T = Id + \epsilon S$, then 
\[
\epsilon^{-1}(U - T^*U(T \cdot)) \underset{\epsilon \to 0}{\longrightarrow} - (Sx \cdot \nabla)U(x) - S^*U(x).
\]
Also if we define $\lambda = \epsilon^{-2}$; $T_+ = Id + \epsilon S$ and $T_- = Id - \epsilon S$, then once again formally :
\[
\epsilon^{-2}(U - T_+^*U(T_+ \cdot) + U - T_-^*U(T_- \cdot)) \underset{\epsilon \to 0}{\longrightarrow} - 2 S^*DU(x).Sx -  D^2U.Sx.Sx
\]
Thus the propagation of monotonicity shall hold for these two equations :
\begin{equation}\label{1storder}
\begin{cases}
\partial_t U(t,x) + (F(x, U) \cdot \nabla)U - (Sx \cdot \nabla)U(x) - S^*U(x) = G(x,U) \text{ in } (0,t_f)\times \mathbb{R^d};\\
U(0,x) = U_0(x) \text{ in } \mathbb{R^d};
\end{cases}
\end{equation}

\begin{equation}\label{2ndorder}
\begin{cases}
\partial_t U(t,x) + (F(x, U) \cdot \nabla)U -  S^*(Sx \cdot \nabla )U(x) - 2(Sx \cdot \nabla^2 \cdot Sx) U = G(x,U) \text{ in } (0,t_f)\times \mathbb{R^d};\\
U(0,x) = U_0(x) \text{ in } \mathbb{R^d}.
\end{cases}
\end{equation}

To avoid going into technical difficulties we do not address here the question of the regularity of (\ref{1storder}) and (\ref{2ndorder}). We restrict ourselves to the proof of the fact that the monotonicity indeed propagates, as it is expected. The following result presents this fact.

\begin{Theorem}
Assume that $U_0$ and $(G,F)$ are monotone and let $U$ be a solution of either \eqref{1storder} or \eqref{2ndorder}, then $U$ is monotone for all time.
\end{Theorem}
\begin{proof}
We begin with the case of \eqref{1storder}. We consider $U$ and $V$ two solutions of \eqref{1storder}. As we did in the proof of \ref{monotone1}, we define $W$ by
\[
W(t,x,y) = < U(t,x) - V(t,y), x - y>.
\]
Let us remark that $W$ satisfies :
\[
\begin{aligned}
\partial_t W +  & F(x,U) \cdot \nabla_x W + F(y,V)\cdot \nabla_y W - <(\nabla_x W - \nabla_y W), Sx -Sy> \\
& = <G(x,U) - G(y,V), x - y> + < F(x,U) - F(y, V), U(t,x) - V(t,y)>.
\end{aligned}
\]
Hence, concluding as we did in proposition \ref{monotone1}, if $W$ is positive at the initial time, it stays positive for all time and thus the propagation of monotonicity is proved. If $U$ and $V$ are solutions of \eqref{2ndorder}, then $W$ satisfies :
\[
\begin{aligned}
\partial_t W +  & F(x,U) \cdot \nabla_x W + F(y,V)\cdot \nabla_y W - Sx \nabla_{xx} W Sx  - Sy\nabla_{yy} W Sy \\
& = <G(x,U) - G(y,V), x - y> + < F(x,U) - F(y, V), U(t,x) - V(t,y)>.
\end{aligned}
\]
Once again, we can conclude with lemma \ref{max} (in appendix) that if $W$ is positive at the initial time it remains positive for all time.
\end{proof}

\section{ From MFG to agent based models}
In this last part, we investigate to which extent MFG models can approach agent based models (ABM). The terminology ABM is often used in Economics and other sciences to refer to a model in which macroscopic phenomena are captured by aggregating individual actions. The main difference with MFG models is that in the ABM, the agents do not anticipate on the behavior of the other agents, hence they are not entirely rational but obey a given "mechanical" rule. For this reason, ABM are purely forward models. The main idea of this section is that one can measure the weight of anticipation for a given agent in a MFG model through the inter temporal preference parameter, thus taking this parameter to infinity should allow to recover an agent-based model from a MFG model. Indeed the players will no longer anticipate as they have no interest in the future. In the following MFG system the inter temporal preference parameter is always denoted by $\lambda$ (which has nothing to do with the parameter $\lambda$ in the part on the common noise) :

\begin{equation}\label{mfglambda}
\begin{cases}
-\partial_t u - \nu \Delta u + H(t,x, \nabla u, m) + \lambda u = 0 \text{ in } (0,T) \times \mathbb{R}^d;\\
\partial_t m - \nu \Delta m - div(D_pH(t,x,\nabla u, m)m) = 0 \text{ in } (0,T)\times \mathbb{R}^d;\\
u(T) = \phi(m(T)) ; m(0) = m_0 \text{ in } \mathbb{R}^d.
\end{cases}
\end{equation}
In this system, $H(x, p , m)$ still denotes the Hamiltonian of the optimal control problem faced by the players. As $\lambda$ increases, the players are more and more interested in the present over the future. The main idea of this section is that letting $\lambda$ go to infinity, one expects that $u$ goes to $0$ and that $(\ref{mfglambda})$ reduces to
\begin{equation}\label{agentbased}
\begin{cases}
\partial_t m - \nu \Delta m - div(D_pH(t,x,0, m)m) = 0 \text{ in } (0,T)\times \mathbb{R}^d; \\
m(0) = m_0.
\end{cases}
\end{equation}
In the first section we work in a particular case and show how we can establish such a convergence. In the second section, we indicate some extensions of this model without giving proofs. Because this remark on MFG is quite general, we prefer to explain the approach on a simple example rather than trying to obtain the most general results.

\subsection{An example of convergence results}
We here give some assumptions we make on the Hamiltonian for this section :
\begin{itemize}
\item $H$ is Lipschitz in $x$, $t$ and $p$, uniformly in $m$.
\item $H : m \rightarrow ((t,x,p) \rightarrow H(t,x,p,m))$ is continuous from $\mathcal{C}((0,T), \mathcal{P})$ to the set of Lipschitz functions. 
\end{itemize}
The following result captures the idea that MFG models approach ABM.

\begin{Theorem}\label{convergence}
For any $\lambda > 0$, we denote by $(u_{\lambda}, m_{\lambda}) \in C^{1,2, \alpha} \times \mathcal{C}((0,T), \mathcal{P})$ a solution of (\ref{mfglambda}) (with $0 < \alpha < 1$). This sequence is bounded and any limit point $m_*$ of the sequence $(m_{\lambda})_{\lambda > 0}$ is a weak solution of :
\begin{equation*}
\begin{cases}
\partial_t m - \nu \Delta m - div(D_pH(t,x,0, m)m) = 0 \text{ in } (0,T)\times \mathbb{R}^d;\\
m(0) = m_0 \text{ in } \mathbb{R}^d.
\end{cases}
\end{equation*}
\end{Theorem}

\begin{proof}
From classical estimates on Hamilton-Jacobi-Bellman equations and on Fokker-Planck equations, we know that solutions of $(\ref{mfglambda})$ exist for all $\lambda > 0$, as a consequence of Schauder's fixed point theorem. For any sequence $(u_{\lambda}, m_{\lambda})_{\lambda}$ of solutions of (\ref{mfglambda}), $(u_{\lambda})_{\lambda > 0}$ is bounded in $L^2(H^2)$. Indeed, if we multiply the equation satisfied by $u_{\lambda}$ by $-\Delta u_{\lambda}$ and we integrate, we obtain 
\begin{equation}
-\int_t^T\int_{\mathbb{R}^d} \frac{1}{2} \frac{d}{dt} |\nabla u_{\lambda}|^2 + \nu \int_t^T \int_{\mathbb{R}^d} (\Delta u_{\lambda})^2 + \lambda \int_t^T \int_{\mathbb{R}^d} |\nabla u_{\lambda} |^2 = \int_t^T \int_{\mathbb{R}^d} \Delta u_{\lambda} H(s,x, \nabla u_{\lambda}, m_{\lambda})dxds ;
\end{equation}
we then deduce
\[
\nu \int_t^T \int_{\mathbb{R}^d} (\Delta u_{\lambda})^2 +\frac{1}{2}\int_{\mathbb{R}^d}|\nabla u_{\lambda}|^2 (t) \leq \frac{1}{2}\int_{\mathbb{R}^d}|\nabla u_{\lambda}|^2 (T) +  \int_t^T \int_{\mathbb{R}^d} \Delta u_{\lambda} H(s,x, \nabla u_{\lambda}, m_{\lambda})dxds .
\]
We obtain from the growth assumption on $H$ that $(\Delta u_{\lambda})_{\lambda}$ is a bounded sequence of $L^2((0,T) \times \mathbb{R}^d)$. Thus, $(u_{\lambda})_{\lambda}$ converges to $0$ in $L^2((0,T), H^1(\mathbb{R}^d))$. Because $(D_pH(\cdot,\cdot, \nabla u_{\lambda},m_{\lambda}))_{\lambda}$ is a bounded sequence in $L^{\infty}$, $(m_{\lambda})_{\lambda}$ is a bounded sequence of $\mathcal{C}((0,T), \mathcal{P})$ where we have equipped $\mathcal{P}$ with the Monge-Kantorovich distance. Hence, by passing to the limit in the weak formulation of the equation satisfied by $m_{\lambda}$, we deduce that any limit point of this sequence is a weak solution of the Fokker-Planck equation (\ref{agentbased}).
\end{proof}

We now show that an ABM can be approximated by a MFG system. The result we present here is that, if it is regular enough, there is always at least one solution of the agent based model which is the limit of a MFG. Indeed let us take an ABM defined by a smooth vector field $B$ such that  :
\[
\exists C >0, \forall t \in (0,T), \forall x \in \mathbb{R}^d, \forall m \in \mathcal{C}((0,T), \mathcal{P}), ||B(t,x,m) ||_{L^{\infty}} \leq C.
\] 
The corresponding so called agent-based model is then :
\begin{equation}\label{exagbm}
\begin{cases}
\partial_t m - \nu \Delta m - div(B(t,x,m)m) = 0 \text{ in } (0,T)\times \mathbb{R}^d;\\
m(0) = m_0.\\
\end{cases}
\end{equation}

Let us take any $\delta > 0$ and define by $H$ the following hamiltonian :
\[
H(t,x,p,m) = B(t,x,m)\cdot p +  \delta|p|^2.
\]

\begin{Prop}\label{ABMtoMFG}
For any sequence of solutions $(u_{\lambda}, m_{\lambda})_{\lambda}$ of :
\begin{equation}\label{approxsystem}
\begin{cases}
-\partial_t u - \nu \Delta u + H(t,x, \nabla u, m) + \lambda u = 0 \text{ in } (0,T)\times \mathbb{R}^d;\\
\partial_t m - \nu \Delta m - div(D_pH(t,x,\nabla u, m)m) = 0 \text{ in } (0,T)\times \mathbb{R}^d;\\
u(T) = 0 ; m(0) = m_0 \text{ in } \mathbb{T}^d;
\end{cases}
\end{equation}
any accumulating point of the bounded sequence $(m_{\lambda})_{\lambda}$ is a solution of the ABM (\ref{exagbm}).
\end{Prop}

\begin{proof}
In view of theorem \ref{convergence}, the sequence of solutions of the MFG system (\ref{approxsystem}) is compact and its accumulating points are solutions of the agent-based model(\ref{exagbm}).
\end{proof}

\subsection{More general models}
In this section we give examples of how one can apply the previous remark to more general models. The first one is a MFG, where a part of the running cost has a proportion of the total cost, which varies with the parameter $\lambda$. The second one is a higher order approximation of the limit model. We conclude with cases in which a common noise can occur.

\subsubsection{Relative running cost}
 We here present the limit as $\lambda$ goes to infinity of the following MFG system :
\begin{equation}\label{psim}
\begin{cases}
-\partial_t u - \nu \Delta u + H(x, \nabla u, m) + \lambda (u - \psi(m)) = 0 \text{ in } (0,T) \times \mathbb{R}^d;\\
\partial_t m - \nu \Delta m - div(mD_pH(x,\nabla u, m)) = 0 \text{ in } (0,T)\times \mathbb{R}^d;\\
u(T) = \phi(m(T)) ; m(0) = m_0 \text{ in } \mathbb{R}^d;
\end{cases}
\end{equation}
where $\psi$ is a smooth function of $m$. The interpretation of such a system, in particular of such an HJB equation, is that $\psi(m)$ stands for a cost which has an increasing importance for the players as $\lambda$ goes to infinity. This cost should be interpreted as a cost the players anticipate with the same weight, whatever the length of the game is. The ABM limit of (\ref{psim}) is the following :
\begin{equation}\label{abmpsim}
\begin{cases}
\partial_t m - \nu \Delta m -div(mD_pH(x, \nabla_x(\psi(m)),m)) = 0 \text{ in } (0,T)\times\mathbb{R}^d ;\\
m(0) = m_0 \text{ in } \mathbb{R}^d.
\end{cases}
\end{equation}
Let us remark that from the results on uniqueness in MFG (\citep{lasry2007mean,lions2007cours}), the usual conditions under which (\ref{psim}) has a unique solution, implies in particular that $\psi$ is monotone in $m$ and $H$ is convex in $p$. On the other hand, the higher order term in the first line of (\ref{abmpsim}) is (in the case of a local $\psi$ to simplify notation):
\[
- \nu \Delta m - m\psi'(m) tr(D_{pp}H(x, \nabla_x(\psi(m)), m) \cdot \nabla_{xx}^2m);
\]
where $tr$ stands for the trace operator. Hence, if the Hamiltonian $H$ is convex in $p$ and $\psi$ is monotone in $m$, we deduce that (\ref{abmpsim}) is a non-linear parabolic equation. This remark makes an obvious link between the well-posedness of the two models and the monotonicity of $\psi$.

\subsubsection{Higher order approximation}
We present here a higher order approximation of the Fokker-Planck equation in (\ref{mfglambda}) in the limit $\lambda$ goes to infinity. Let us observe that, at least formally, we have the following in $(0,T) \times \mathbb{T}^d$:
\begin{equation*}\label{devu}
u(t,x) \underset{\lambda \rightarrow + \infty}{\approx} \frac{-1}{\lambda} H(x,0,m) + O(\lambda^{-2}).
\end{equation*}
Thus if we neglect terms of order $\lambda^{-2}$, the resulting Fokker-Planck equation is :
\begin{equation}\label{FPorder1}
\begin{cases}
\partial_t m - \nu \Delta m -div(mD_pH(x, \nabla_x(\frac{-1}{\lambda}H(x,0,m)),m)) = 0 \text{ in } (0,T)\times\mathbb{R}^d ;\\
m(0) = m_0 \text{ in } \mathbb{R}^d.
\end{cases}
\end{equation}
As we did for the relative running cost, we show that under the usual condition of uniqueness for (\ref{mfglambda}), the equation (\ref{FPorder1}) is parabolic. The higher order term in $m$ in (\ref{FPorder1}) is, when $H$ is local in $m$ :
\begin{equation*}
- \nu \Delta m + \frac{1}{\lambda}mD_zH(x,0,m) tr(D_{pp}H(x,\nabla_x(\frac{-1}{\lambda}H(x,0,m)),m)\cdot \nabla^2_{xx}m);
\end{equation*}
where $tr$ still stands for the trace operator and $D_zH$ for the derivative of $H$ with respect to its third variable. We recall here the usual conditions under which (\ref{mfglambda}) has a unique solution (\citep{lions2007cours}) : for any $(x,p,z) \in \mathbb{T}^d \times \mathbb{R}^d \times \mathbb{R}_+$ :
\[
\begin{pmatrix}
zD_{pp}H(x,p,z) & \frac{z}{2}D_{pz}H(x,p,z) \\
\frac{z}{2}D_{pz}H(x,p,z) & - D_zH(x,p,z)
\end{pmatrix} \geq 0.
\] 
Thus, as in the previous case, the uniqueness of solutions of (\ref{mfglambda}) implies the well-posedness of (\ref{FPorder1}). This strongly suggests that the previous assumption on the derivative of $H$ is close from being necessary for the uniqueness in (\ref{mfglambda}).

\subsubsection{The case of common noise : discrete state space}
We now explain why, formally, the addition of a common noise should not alter the property that, in the limit $\lambda$ goes to $\infty$, a MFG converges toward an ABM. We present first a discrete state space case (cf section \ref{sec : mfgn}), and next a continuous state space case with the notations of \citep{cardaliaguet2015master}.\\
\\
For the discrete state space case, we work with a common noise described by an intensity of jump $\beta > 0$ and a linear map $T$. The map $T$ is assumed to be linear for the sake of clarity of the derivation of the equation (\ref{eqtildeu}), but this result apply to more general maps $T$. We introduce the solution $U$ of 
\begin{equation}\label{MFGnld}
\begin{cases}
\partial_t U(t,x) + (F(x, U) \cdot \nabla)U + \lambda U+  \beta (U - T^*U(t, Tx)) = G(x,U) \text{ in } (0,t_f)\times \mathbb{R^d};\\
U(t_f,x) = U_0(x) \text{ in } \mathbb{R^d}.
\end{cases}
\end{equation}
We fix a probability space $(\Omega, \mathcal{A}, \mathbb{P})$. We denote by  $(\tau_i)_{i \in \mathbb{N}}$ a sequence of independent random variables of exponential law of parameter $\beta$. We define the sequence $(t_i)_{i \in \mathbb{N}}$ by 
\[
t_i = \sum_{k = 0}^i \tau_k.
\]
We define for any $t > 0$, $x_0 \in \mathbb{R}^d$ the couple $(V_s, Y_s)_{0 \leq s \leq t}$ by the following :
\begin{equation}\label{chara2}
\begin{cases}
d(e^{\lambda s}V_s) = G(Y_s, U(s, Y_s)) ds ;  \forall s \in (t_i, t_{i+1}) ;\\
dY_s = F(Y_s, U(s, Y_s))ds ;  \forall s \in (t_i, t_{i+1}) ; \\
Y_{t_i^+} = TY_{t_i^-} ; V_{t_i^+} = T^*V_{t_i^-}\\
Y_t = x_0 ; V_0 = U_0(Y(0)).
\end{cases}
\end{equation}
This couple $(V_s, Y_s)$ represents the characteristic associated to (\ref{MFGnld}) in the sense that if we defined $\tilde{U}$ by 
\[
\tilde{U}(t,x_0) = \mathbb{E}[V_t]
\]
with $(V_s, Y_s)_{s \geq 0}$ being defined by (\ref{chara2}), then a formal calculation yields that $\tilde{U}$ solves :
\begin{equation}\label{eqtildeu}
\begin{cases}
\partial_t \tilde{U}(t,x) + (F(x, U) \cdot \nabla)\tilde{U} + \lambda \tilde{U}+  \beta (\tilde{U} - T^*\tilde{U}(t, Tx)) = G(x,U) \text{ in } (0,t_f)\times \mathbb{R^d};\\
U(0,x) = U_0(x) \text{ in } \mathbb{R^d}.
\end{cases}
\end{equation}
Thus formally $\tilde{U} = U$. Hence, the process $(Y_s)_{s \geq 0}$ should represent a realization of the evolution of a population of player facing the MFG modeled by (\ref{MFGnld}) and starting from the configuration $x_0$. Formally, in the limit $\lambda$ goes to infinity, we obtain that that $U$ goes to $0$ uniformly in $x$, and thus the evolution of the process $(Y_s)_{s \geq 0}$ is given by :
\begin{equation}\label{ABMd}
\begin{cases}
dY_s = F(Y_s, 0)ds ;  \forall s \in (t_i, t_{i+1}) ; \\
Y_0 = x_0 ; Y_{t_i^+} = TY_{t_i^-} ;
\end{cases}
\end{equation}
where the sequence $(t_i)_{i \in \mathbb{N}}$ has the same law as in (\ref{chara2}). Hence, the "model" (\ref{ABMd}) is the limit ABM one should obtain in the limit $\lambda$ goes to infinity for (\ref{MFGnld}).

\subsubsection{The case of common noise : the infinite dimensional case}
Assume now that the MFG is in infinite dimension and that the master equation is given by :
\begin{equation}\label{masterequation}
\begin{cases}
\begin{aligned}
&-\partial_t U - (1+ \beta) \Delta_x U + H(x, D_x U) - (1 + \beta)\int_{\mathbb{R}^d}div_y [D_m U] dm(y)\\
&+ \int_{\mathbb{R}^d}D_m U \cdot D_p H(y, D_x U)dm(y)- 2\beta \int_{\mathbb{R}^d}dix_x[D_m U]dm(y) \\
&- \beta \int_{\mathbb{R}^{2d}}Tr[D^2_{mm}U]dm \otimes dm + \lambda U= F(x,m) \text{ in } [0,T] \times \mathbb{R}^d \times \mathcal{P}(\mathbb{R}^d);
\end{aligned}\\
U(T,x,m) = G(x, m) \text{ in } \mathbb{R}^d \times \mathcal{P}(\mathbb{R}^d).
\end{cases}
\end{equation}
In this setting $H$ is the hamiltonian of the optimal control problem faced by the players, $F$ the running cost and $G$ the terminal cost. We refer to \citep{cardaliaguet2015master} for a precise description of the MFG this master equation models and for definitions of the derivatives in the space of measures. We just recall here that $\beta$ is positive parameter which describes the intensity of the common noise ; that $F$ and $G$ satisfies some strong regularity assumptions and that $H$ is smooth, globally Lipschitz and coercive in its second argument. We denote by $(\Omega, \mathcal{A}, \mathbb{P})$ a probability space. In the context modeled by (\ref{masterequation}), the forward-backward system (\ref{MFG}) takes the form :
\begin{equation}\label{MFGnoise}
\begin{cases}
d_t u_t= \{ -(1 + \beta)\Delta u_t + H(x, D_x u_t) + \lambda u_t - F(x, m_t) - \sqrt{2 \beta}div(v_t)\} dt + v_t \cdot \sqrt{2 \beta} d W_t \\
\text{ in } (0,T) \times \mathbb{R}^d;\\
d_t m_t = \{(1+ \beta)\Delta m_t + div(m_t D_p H(x, D_x u_t))\} dt - div(m_t\sqrt{2\beta}W_t) \text{ in } (0,T) \times \mathbb{R}^d;\\
m_0 = m_0 ; u_T = G(x, m_T) \text{ in } \mathbb{R}^d.
\end{cases}
\end{equation}
The process $(W_t)_{0 \leq t \leq T}$ is a standard $d$ dimensional bronwnian motion under the probability space $(\Omega, \mathcal{A}, \mathbb{P})$. The unknown are the three random functions $(u_t, m_t, v_t)_{0 \leq t \leq T}$ and $(v_t)_{0 \leq t \leq T}$ is a term which forces $u_t$ to be measurable with respect to the filtration generated by $(W_t)_{0 \leq t}$. Formally as $\lambda$ goes to infinity, the solution $U$ of (\ref{masterequation}) goes to $0$ and thus (\ref{MFGnoise}) reduces to :
\begin{equation}
\begin{cases}
d_t m_t = \{(1+ \beta)\Delta m_t + div(m_t D_p H(x, 0))\} dt - div(m_t\sqrt{2\beta}W_t) \text{ in } (0,T) \times \mathbb{R}^d;\\
m_0 = m_0 \text{ in } \mathbb{R}^d.
\end{cases}
\end{equation}
This nonlinear stochastic Fokker-Planck equation is the limit ABM.

\bibliographystyle{plainnat}
\bibliography{bibremarks}

\begin{thebibliography}{26}
\providecommand{\natexlab}[1]{#1}
\providecommand{\url}[1]{\texttt{#1}}
\expandafter\ifx\csname urlstyle\endcsname\relax
  \providecommand{\doi}[1]{doi: #1}\else
  \providecommand{\doi}{doi: \begingroup \urlstyle{rm}\Url}\fi

\bibitem[Achdou and Capuzzo-Dolcetta(2010)]{achdou2010mean}
Yves Achdou and Italo Capuzzo-Dolcetta.
\newblock Mean field games: Numerical methods.
\newblock \emph{SIAM Journal on Numerical Analysis}, 48\penalty0 (3):\penalty0
  1136--1162, 2010.

\bibitem[Benamou et~al.(2017)Benamou, Carlier, and
  Santambrogio]{benamou2017variational}
Jean-David Benamou, Guillaume Carlier, and Filippo Santambrogio.
\newblock Variational mean field games.
\newblock In \emph{Active Particles, Volume 1}, pages 141--171. Springer, 2017.

\bibitem[Bertucci(2017)]{bertucci2017optimal}
Charles Bertucci.
\newblock Optimal stopping in mean field games, an obstacle problem approach.
\newblock \emph{arXiv preprint arXiv:1704.06553}, 2017.

\bibitem[Bertucci(2018)]{bertucci2018fokker}
Charles Bertucci.
\newblock Fokker-planck equation of jumping particles and mean field games of
  impulse control.
\newblock \emph{to appear}, 2018.

\bibitem[Briceno-Arias et~al.(2016)Briceno-Arias, Kalise, and
  Silva]{briceno2016proximal}
LM~Briceno-Arias, D~Kalise, and FJ~Silva.
\newblock Proximal methods for stationary mean field games with local
  couplings.
\newblock \emph{arXiv preprint arXiv:1608.07701}, 2016.

\bibitem[Cardaliaguet and Hadikhanloo(2015)]{cardaliaguet2015learning}
Pierre Cardaliaguet and Saeed Hadikhanloo.
\newblock Learning in mean field games: the fictitious play.
\newblock \emph{arXiv preprint arXiv:1507.06280}, 2015.

\bibitem[Cardaliaguet and Lehalle(2016)]{cardaliaguet2016mean}
Pierre Cardaliaguet and Charles-Albert Lehalle.
\newblock Mean field game of controls and an application to trade crowding.
\newblock \emph{arXiv preprint arXiv:1610.09904}, 2016.

\bibitem[Cardaliaguet and Porretta(2017)]{cardaliaguet2017long}
Pierre Cardaliaguet and Alessio Porretta.
\newblock Long time behavior of the master equation in mean-field game theory.
\newblock \emph{arXiv preprint arXiv:1709.04215}, 2017.

\bibitem[Cardaliaguet et~al.(2012)Cardaliaguet, Lasry, Lions, Porretta,
  et~al.]{cardaliaguet2012long}
Pierre Cardaliaguet, Jean-Michel Lasry, Pierre-Louis Lions, Alessio Porretta,
  et~al.
\newblock Long time average of mean field games.
\newblock \emph{NHM}, 7\penalty0 (2):\penalty0 279--301, 2012.

\bibitem[Cardaliaguet et~al.(2015{\natexlab{a}})Cardaliaguet, Delarue, Lasry,
  and Lions]{cardaliaguet2015master}
Pierre Cardaliaguet, Fran{\c{c}}ois Delarue, Jean-Michel Lasry, and
  Pierre-Louis Lions.
\newblock The master equation and the convergence problem in mean field games.
\newblock \emph{arXiv preprint arXiv:1509.02505}, 2015{\natexlab{a}}.

\bibitem[Cardaliaguet et~al.(2015{\natexlab{b}})Cardaliaguet, Graber, Porretta,
  and Tonon]{cardaliaguet2015second}
Pierre Cardaliaguet, P~Jameson Graber, Alessio Porretta, and Daniela Tonon.
\newblock Second order mean field games with degenerate diffusion and local
  coupling.
\newblock \emph{Nonlinear Differential Equations and Applications NoDEA},
  22\penalty0 (5):\penalty0 1287--1317, 2015{\natexlab{b}}.

\bibitem[Carmona and Delarue(2017{\natexlab{a}})]{carmona2017probabilistic}
R~Carmona and F~Delarue.
\newblock Probabilistic theory of mean field games: vol. i, mean field fbsdes,
  control, and games.
\newblock \emph{Stochastic Analysis and Applications. Springer Verlag},
  2017{\natexlab{a}}.

\bibitem[Carmona and Delarue(2017{\natexlab{b}})]{carmona2017probabilistic2}
R~Carmona and F~Delarue.
\newblock Probabilistic theory of mean field games: vol. ii, mean field games
  with common noise and master equations.
\newblock \emph{Stochastic Analysis and Applications. Springer Verlag},
  2017{\natexlab{b}}.

\bibitem[Carmona and Delarue(2013)]{carmona2013probabilistic}
Ren{\'e} Carmona and Fran{\c{c}}ois Delarue.
\newblock Probabilistic analysis of mean-field games.
\newblock \emph{SIAM Journal on Control and Optimization}, 51\penalty0
  (4):\penalty0 2705--2734, 2013.

\bibitem[Carmona et~al.(2015)Carmona, Lacker, et~al.]{carmona2015probabilistic}
Ren{\'e} Carmona, Daniel Lacker, et~al.
\newblock A probabilistic weak formulation of mean field games and
  applications.
\newblock \emph{The Annals of Applied Probability}, 25\penalty0 (3):\penalty0
  1189--1231, 2015.

\bibitem[Carmona et~al.(2016)Carmona, Delarue, and Lacker]{carmona2016mean}
Rene Carmona, Fran{\c{c}}ois Delarue, and Daniel Lacker.
\newblock Mean field games of timing and models for bank runs.
\newblock \emph{arXiv preprint arXiv:1606.03709}, 2016.

\bibitem[Conze et~al.(1991)Conze, Lasry, and Scheinkman]{conze1991system}
A~Conze, J-M Lasry, and JA~Scheinkman.
\newblock A system of non-linear functional differential equations arising in
  an equilibrium model of an economy with borrowing constraints.
\newblock In \emph{Annales de l'Institut Henri Poincare (C) Non Linear
  Analysis}, volume~8, pages 523--559, 1991.

\bibitem[Conze et~al.(1993)Conze, Lasry, and Scheinkman]{conze1993borrowing}
Antoine Conze, Jean-Michel Lasry, and Jose Scheinkman.
\newblock Borrowing constraints and international comovements.
\newblock \emph{Hitotsubashi journal of economics}, pages 23--47, 1993.

\bibitem[Delaure et~al.(2018)Delaure, Lacker, and
  Ramanan]{delarue2018fluctuations}
F~Delaure, D~Lacker, and K~Ramanan.
\newblock From the master equations to mean field games limits, fluctuations
  and large deviations.
\newblock \emph{mien}, 2018.

\bibitem[Gomes et~al.(2014)Gomes, Patrizi, and Voskanyan]{Gomes2014existence}
Diogo~A Gomes, Stefania Patrizi, and Vardan Voskanyan.
\newblock On the existence of classical solutions for stationary extended mean
  field games.
\newblock \emph{Nonlinear Analysis: Theory, Methods \& Applications},
  99:\penalty0 49--79, 2014.

\bibitem[Lacker(2015)]{lacker2015mean}
Daniel Lacker.
\newblock Mean field games via controlled martingale problems: existence of
  markovian equilibria.
\newblock \emph{Stochastic Processes and their Applications}, 125\penalty0
  (7):\penalty0 2856--2894, 2015.

\bibitem[Lasry and Lions(2006{\natexlab{a}})]{lasry2006jeux}
Jean-Michel Lasry and Pierre-Louis Lions.
\newblock Jeux {\`a} champ moyen. i--le cas stationnaire.
\newblock \emph{Comptes Rendus Math{\'e}matique}, 343\penalty0 (9):\penalty0
  619--625, 2006{\natexlab{a}}.

\bibitem[Lasry and Lions(2006{\natexlab{b}})]{lasry2006jeux2}
Jean-Michel Lasry and Pierre-Louis Lions.
\newblock Jeux {\`a} champ moyen. ii--horizon fini et contr{\^o}le optimal.
\newblock \emph{Comptes Rendus Math{\'e}matique}, 343\penalty0 (10):\penalty0
  679--684, 2006{\natexlab{b}}.

\bibitem[Lasry and Lions(2007)]{lasry2007mean}
Jean-Michel Lasry and Pierre-Louis Lions.
\newblock Mean field games.
\newblock \emph{Japanese Journal of Mathematics}, 2\penalty0 (1):\penalty0
  229--260, 2007.

\bibitem[Lions(2007)]{lions2007cours}
Pierre-Louis Lions.
\newblock Cours au college de france.
\newblock \emph{www.college-de-france.fr}, 2011, 2007.

\bibitem[Nutz(2016)]{nutz2016mean}
Marcel Nutz.
\newblock A mean field game of optimal stopping.
\newblock \emph{arXiv preprint arXiv:1605.09112}, 2016.

\end{thebibliography}

\appendix
\section{A maximum principle result}
We present a maximum principle result which is quite general and that we use several times in this paper. Although this result is not new (see \citep{lions2007cours} for example), we recall it and give a simple proof for the sake of completeness.
\begin{Lemma}\label{max}
Let $u$ be a smooth function from $(0,T)\times \mathbb{R}^d$ to $\mathbb{R}$ which satisfies :
\begin{equation}
\begin{cases}
\partial_t u - A(t,x,\nabla_x^2 u) + B(t,x, \nabla_x u) + \alpha(t,x) (u - u(t,f(x))) \geq 0 \text{ in } (0,T) \times \mathbb{R}^d;\\
u(0) \geq 0 \text{ in } \mathbb{R}^d;
\end{cases}
\end{equation}
where $A$ is a uniform elliptic term which is linear in its third argument, $B$ is linear with respect to its third argument, $\alpha \in L^{\infty}((0,T) \times \mathbb{R}^d)$, $\alpha \geq 0$ and $f$ maps $\mathbb{R}^d$ into itself. Then $u \geq 0$ on $(0,T) \times \mathbb{R}^d$.
\end{Lemma}
\begin{proof}
Let us assume that there exists $(t_0, x_0) \in (0,T)\times \mathbb{R}^d$ such that $u(t_0, x_0) < 0$. Thus, for any $\lambda \geq 0$, there exists $\delta > 0$ such that 
\begin{equation}\label{ld}
e^{-\lambda t_0}u(t_0, x_0) + \delta t_0 < 0.
\end{equation}
We define $v$ by 
\[
v(t,x) = e^{-\lambda t} u (t,x) + \delta (t +1).
\]
This function satisfies 
\begin{equation}\label{eqppmax}
\begin{cases}
\partial_t v - A(t,x, \nabla^2_x v) + e^{-\lambda t} B(t,x,\nabla_x u) + \alpha(t,x)(v - v(t,f(x))) - \lambda v - \delta - \lambda \delta (t+1) \geq 0\\
\text{ in } (0,T)\times \mathbb{R}^d \\
v(0) > 0 \text{ in } \mathbb{R}^d\\
v(t_0, x_0) < 0.
\end{cases}
\end{equation}
Thus, there exists $(t_1, x_1)$ such that :
\[
\begin{cases}
v(t_1, x_1) = 0 \\
\partial_t v(t_1, x_1) \leq 0\\
\nabla_x^2 v(t_1, x_1) \geq 0 \\
\nabla_x v(t_1, x_1) = 0\\
v(t, x) \geq 0,  \forall t \leq t_1, \forall x.
\end{cases}
\]
Because $\nabla_x v (t_1, x_1) = 0 $ implies that $\nabla_x u(t_1, x_1) = 0$, evaluating the PDE satisfied by $u$, we deduce that
\begin{equation*}\label{relationcontr}
- \delta (1 + \lambda (t_1+1)) \geq 0.
\end{equation*}
Thus we obtain the desired contradiction.
\end{proof}

\section{ Proof of Theorem \ref{amonotone} under general assumptions}

The proof of Theorem \ref{amonotone} follows the argument from \citep{lions2007cours}.  We show first the case $U_0$ and $G$ are $\alpha$ monotone, and then the case $F$ is $\alpha$ monotone.
\subsection{$U_0$ and $G$ are $\alpha$ monotone}
As we did above, we define $W$ and $Z_{\beta}$ by
\begin{equation}\label{defW}
W(t,x,\xi) = <U(t,x), \xi> ;
\end{equation}
\[
Z_{\beta}(t,x,\xi) = <\xi, \nabla_x W(t,x,\xi)> - \beta(t) |\nabla_x W(t,x,\xi)|^2;
\]
except that now $\beta$ can be a function of the time. The PDE satisfied by $Z_{\beta}$ is :

\begin{equation}\label{eqzb2}
\begin{aligned}
& \partial_t Z_{\beta} + <F(x,\nabla_{\xi} W), \nabla_x Z_{\beta}> + < D_p F (x,\nabla_{\xi} W ) \nabla_{\xi}Z_{\beta} , \nabla_{x} W> - < D_p G(x, \nabla_{\xi}W)\nabla_{\xi}Z_{\beta}, \xi>\\
&+  \lambda(Z_{\beta} - Z_{\beta}(t,Tx, T\xi - e) )\\
& = <\nabla_xG(x, \nabla_x W) \xi, \xi> - < D_pG(x, \nabla_{\xi}W) \nabla_x W, \xi> - < \nabla_x F(x, \nabla_{\xi} W) \nabla_x W, \xi> \\
&+ <D_p F(x, \nabla_{\xi}W) \nabla_x W, \nabla_x W>\\
&- 2 \beta < \nabla_x G(x, \nabla_x W)\xi, \nabla_x W> + 2 \beta < \nabla_x F( \nabla_{\xi} W) \nabla_x W, \nabla_x W>\\
& + \beta \lambda \bigg{(} |\nabla_x W|^2 - 2 < \nabla_x W(t, Tx, T \xi - e), S \nabla_x W> +  | \nabla_x W(t,Tx, T \xi -e)|^2 \bigg{)} \\
&- \frac{d}{dt} \beta |\nabla_x W|^2.
\end{aligned}
\end{equation}
This equation reduces to (using this time only the $\alpha$ monotonicity of $G$) :
\begin{equation}
\begin{aligned}
& \partial_t Z_{\beta} + <F(x,\nabla_{\xi} W), \nabla_x Z_{\beta}> + < D_p F (x,\nabla_{\xi} W ) \nabla_{\xi}Z_{\beta} , \nabla_{x} W> - < D_p G(x, \nabla_{\xi}W)\nabla_{\xi}Z_{\beta}, \xi>\\
&+  \lambda(Z_{\beta} - Z_{\beta}(t,Tx, T\xi - e) )\\
& \geq \alpha | \xi |^2 - 2 \beta < \nabla_x G(x, \nabla_x W)\xi, \nabla_x W> + 2 \beta < \nabla_x F( \nabla_{\xi} W) \nabla_x W, \nabla_x W>\\
& + \beta \lambda ( |\nabla_x W|^2 - |S\nabla_x W|^2 ) - \frac{d}{dt} \beta |\nabla_x W|^2.
\end{aligned}
\end{equation}
We then deduce that :
\begin{equation}
\begin{aligned}
& \partial_t Z_{\beta} + <F(x,\nabla_{\xi} W), \nabla_x Z_{\beta}> + < D_p F (x,\nabla_{\xi} W ) \nabla_{\xi}Z_{\beta} , \nabla_{x} W> - < D_p G(x, \nabla_{\xi}W)\nabla_{\xi}Z_{\beta}, \xi>\\
&+  \lambda(Z_{\beta} - Z_{\beta}(t,Tx, T\xi - e) )\\
& \geq \alpha | \xi |^2 -  \beta || \nabla_x G|| \cdot |\xi|^2 - \beta ||\nabla_xG||\cdot |\nabla_x W|^2 - 2 \beta || \nabla_x F( \nabla_{\xi} W)|| \cdot |\nabla_x W|^2\\
& + \beta \lambda ( |\nabla_x W|^2 - |S\nabla_x W|^2 ) - \frac{d}{dt} \beta |\nabla_x W|^2.
\end{aligned}
\end{equation}
Thus if we define $\beta$ by :
\begin{equation}
\beta(t) = \alpha e^{-(2||\nabla_xF|| + ||\nabla_x G|| + (||S||^2 -1)^+ \lambda)t};
\end{equation}
then $Z_{\beta}$ is positive for all time because of lemma \ref{max}. We then conclude the proof as we did in the case in which $U_0$, $G$ and $F$ are $\alpha$ monotone.

\subsection{ $F$ is $\alpha$ monotone}
We still define $W$ by (\ref{defW}) and we define $Z_{\beta, \gamma}$ by :
\[
Z_{\beta}(t,x,\xi) = <\xi, \nabla_x W(t,x,\xi)> - \beta(t) |\nabla_x W(t,x,\xi)|^2 + \gamma(t)|\xi|^2.
\]
Here, both $\beta$ and $\gamma$ are function of the time we shall define later. The PDE satisfied by $Z_{\beta, \gamma}$ is :

\begin{equation}
\begin{aligned}
& \partial_t Z_{\beta, \gamma} + <F(x,\nabla_{\xi} W), \nabla_x Z_{\beta, \gamma}> + < D_p F (x,\nabla_{\xi} W ) \nabla_{\xi}Z_{\beta, \gamma} , \nabla_{x} W> - < D_p G(x, \nabla_{\xi}W)\nabla_{\xi}Z_{\beta, \gamma}, \xi>\\
&+  \lambda(Z_{\beta, \gamma} - Z_{\beta, \gamma}(t,Tx, T\xi - e) )\\
& = <\nabla_xG(x, \nabla_x W) \xi, \xi> - < D_pG(x, \nabla_{\xi}W) \nabla_x W, \xi> - < \nabla_x F(x, \nabla_{\xi} W) \nabla_x W, \xi> \\
&+ <D_p F(x, \nabla_{\xi}W) \nabla_x W, \nabla_x W>\\
&- 2 \beta < \nabla_x G(x, \nabla_x W)\xi, \nabla_x W> + 2 \beta < \nabla_x F( \nabla_{\xi} W) \nabla_x W, \nabla_x W>\\
& + \beta \lambda \bigg{(} |\nabla_x W|^2 - 2 < \nabla_x W(t, Tx, T \xi - e), S \nabla_x W> +  | \nabla_x W(t,Tx, T \xi -e)|^2 \bigg{)} \\
&- \frac{d}{dt} \beta |\nabla_x W|^2  + 2 \gamma(< D_p F (x,\nabla_{\xi} W ) \xi , \nabla_{x} W> - < D_p G(x, \nabla_{\xi}W)\xi, \xi>)\\
& + \frac{d}{dt}\gamma |\xi|^2 + \lambda \gamma (|\xi|^2 - |T\xi - e|^2).
\end{aligned}
\end{equation}
Thus, we deduce the inequality :
\begin{equation}\label{ineqz3}
\begin{aligned}
& \partial_t Z_{\beta, \gamma} + <F(x,\nabla_{\xi} W), \nabla_x Z_{\beta, \gamma}> + < D_p F (x,\nabla_{\xi} W ) \nabla_{\xi}Z_{\beta, \gamma} , \nabla_{x} W> - < D_p G(x, \nabla_{\xi}W)\nabla_{\xi}Z_{\beta, \gamma}, \xi>\\
&+  \lambda(Z_{\beta, \gamma} - Z_{\beta, \gamma}(t,Tx, T\xi - e) )\\
& \geq \alpha | \nabla_x W |^2 -  \beta || \nabla_x G|| \cdot |\xi|^2 - \beta ||\nabla_xG||\cdot |\nabla_x W|^2 - 2 \beta || \nabla_x F( \nabla_{\xi} W)|| \cdot |\nabla_x W|^2\\
& + \beta \lambda ( |\nabla_x W|^2 - |S\nabla_x W|^2 ) - \frac{d}{dt} \beta |\nabla_x W|^2 - \gamma (||D_pF|| - 2||D_p G||) |\xi|^2\\
& - \gamma ||D_p F||\cdot |\nabla_x W |^2 + \frac{d}{dt}\gamma |\xi|^2 + \lambda \gamma (|\xi|^2 - |T\xi - e|^2).
\end{aligned}
\end{equation}
Hence, if $\beta$ and $\gamma$ satisfies
\begin{equation}
\begin{cases}
\alpha - \beta [||\nabla_x G|| + 2 ||\nabla_x F|| - \lambda (1 - ||S||^2)] - \frac{d}{dt}\beta - \gamma ||D_pF|| \geq 0 ;\\
\frac{d}{dt} \gamma + \gamma[ \lambda(1 - ||S||^2) - ||D_pF|| - 2 ||D_pG||] - \beta ||\nabla_x G|| \geq 0;
\end{cases}
\end{equation}
then the right hand side of (\ref{ineqz3}) is positive for all time $t \in (0, t_f)$. Moreover let us remark that, at the initial time, because $U_0$ is monotone : 
\[
Z_{\beta, \gamma} \geq (\gamma(0) - \beta(0)||D_xU_0||^2) |\xi|^2.
\]
Thus if we define $\beta$ and $\gamma$ by
\begin{equation}
\begin{cases}
\frac{d}{dt}\beta = \alpha - \beta [||\nabla_x G|| + 2 ||\nabla_x F|| + \lambda (1 - ||S||^2)^-]  - \gamma ||D_pF||;\\
\frac{d}{dt} \gamma = \gamma[ -\lambda(1 - ||S||^2)^- + ||D_pF|| + 2 ||D_pG||] + \beta ||\nabla_x G||;\\
\beta(0) > 0 ; \gamma(0) = \beta(0)||D_x U_0||^2;
\end{cases}
\end{equation}
then $Z_{\beta, \gamma}$ is positive at $t = 0$ and the right hand side of (\ref{ineqz3}) is positive for all time. Thus, if $\beta(t) > 0$ for all $t \in [0, t_F]$, then by lemma \ref{max} we conclude that $Z_{\beta, \gamma}$ is positive for all time and the theorem follows immediately. Now let us remark that one can always choose $\beta(0)$ small enough so that for all time $t \in [0, t_f]$ :
\[
\begin{cases}
0 \leq \beta(t) \leq \alpha t + \beta(0);\\
\gamma(t) \leq \beta(0)||D_xU_0||^2 e^{(\lambda(1 - ||S||^2)^- + ||D_pF|| + 2 ||D_pG||)t} + \beta(0)||D_xU_0||^2 \alpha \frac{t^2}{2} + \beta(0)||\nabla_xG||t.
\end{cases}
\]
Thus the theorem is proved.

\end{document}